\documentclass[12pt]{article}

\usepackage{amsmath,amssymb,amsbsy,amsfonts,amsthm,latexsym,amsopn,amstext,
amsxtra,euscript,amscd,mathrsfs}

\begin{document}

\newtheorem{theorem}{Theorem}
\newtheorem{lemma}{Lemma}
\newtheorem{corollaire}{Corollaire}
\newtheorem{remarque}{Remarque}

\newtheorem{proposition}{Proposition}
\newtheorem{corollary}{Corollary}

\theoremstyle{definition}
\newtheorem*{definition}{Definition}
\newtheorem*{remark}{\bf Remark}
\newtheorem*{example}{Example}

\def\cA{\mathcal A}
\def\cB{\mathcal B}
\def\cC{\mathcal C}
\def\cD{\mathcal D}
\def\cE{\mathcal E}
\def\cF{\mathcal F}
\def\cG{\mathcal G}
\def\cH{\mathcal H}
\def\cI{\mathcal I}
\def\cJ{\mathcal J}
\def\cK{\mathcal K}
\def\cL{\mathcal L}
\def\cM{\mathcal M}
\def\cN{\mathcal N}
\def\cO{\mathcal O}
\def\cP{\mathcal P}
\def\cQ{\mathcal Q}
\def\cR{\mathcal R}
\def\cS{\mathcal S}
\def\cT{\mathcal T}
\def\cNnot{\cN^\circ}
\def\cV{\mathcal V}
\def\cW{\mathcal W}
\def\cX{\mathcal X}
\def\cY{\mathcal Y}
\def\cZ{\mathcal Z}

\def\Z{{\mathbb Z}}
\def\P{{\mathbb P}}
\def\R{{\mathbb R}}
\def\F{{\mathbb F}}
\def\N{{\mathbb N}}
\def\C{{\mathbb C}}
\def\Q{{\mathbb Q}}
\def\sC{{\mathscr C}}
\def\sD{{\mathscr D}}
\def\sE{{\mathscr E}}
\def\oop{\overline{\overline p}}
\def\ooq{\overline{\overline q}}

\def\Fp{\F_p}
\def\e{\mathbf{e}}
\def\ep{\mathbf{e}_p}
\def\eps{{\varepsilon}}
\def\mand{\qquad \mbox{and} \qquad}
\def\mor{\qquad \mbox{or} \qquad}
\def\scr{\scriptstyle}
\def\\{\cr}
\def\({\left(}
\def\){\right)}
\def\[{\left[}
\def\]{\right]}
\def\<{\langle}
\def\>{\rangle}
\def\fl#1{\left\lfloor#1\right\rfloor}
\def\rf#1{\left\lceil#1\right\rceil}
\def\le{\leqslant}
\def\ge{\geqslant}
\def\pk{\widehat p_k}
\def\qk{\widehat q_k}

\def\xxx{\vskip5pt\hrule\vskip5pt}

\renewcommand{\limsup}{\mathop{\overline{\rm lim}}}
\renewcommand{\liminf}{\mathop{\underline{\rm lim}}}
\newcommand{\ud}{\mathop{\overline{\delta}}}
\newcommand{\ld}{\mathop{\underline{\delta}}}

\newcommand{\comm}[1]{\marginpar{%
\vskip-\baselineskip %raise the marginpar a bit
\raggedright\footnotesize
\itshape\hrule\smallskip#1\par\smallskip\hrule}}

\title{\sc The Nicolas and Robin inequalities\break with sums of two squares}

\author{
{\sc William D.~Banks} \\
{Department of Mathematics} \\
{University of Missouri} \\
{Columbia, MO 65211 USA} \\
{\tt bbanks@math.missouri.edu} \\
\and
{\sc Derrick N.~Hart} \\
{Department of Mathematics} \\
{University of Missouri} \\
{Columbia, MO 65211 USA} \\
{\tt hart@math.missouri.edu} \\
\and
{\sc Pieter Moree\footnote{Corresponding author}} \\
{Max-Planck-Institut f\"ur Mathematik} \\
{Vivatsgasse 7} \\
{D-53111 Bonn, Germany} \\
{\tt moree@mpim-bonn.mpg.de} \\
\and
{\sc C.~Wesley Nevans} \\
{Department of Mathematics} \\
{University of Missouri} \\
{Columbia, MO 65211 USA} \\
{\tt nevans@math.missouri.edu}}

\pagenumbering{arabic}

\date{}

\maketitle

\begin{abstract}
\noindent In 1984, G.~Robin proved that the Riemann hypothesis is
true if and only if the \emph{Robin inequality}
$\sigma(n)<e^\gamma n\log\log n$ holds for every integer $n>5040$,
where $\sigma(n)$ is the sum of divisors function, and $\gamma$ is
the Euler-Mascheroni constant. We exhibit a broad class of subsets
$\cS$ of the natural numbers such that the Robin inequality holds
for all but finitely many $n\in\cS$. As a special case, we
determine the finitely many numbers of the form $n=a^2+b^2$ that
do not satisfy the Robin inequality. In fact, we prove our
assertions with the \emph{Nicolas inequality}
$n/\varphi(n)<e^{\gamma}\log \log n$; since
$\sigma(n)/n<n/\varphi(n)$ for $n>1$ our results for the Robin
inequality follow at once.
\end{abstract}

\section{Introduction}

Let $\varphi(n)$ denote the \emph{Euler function}. In 1903 Landau
(see~\cite[pp. 217--219]{Land}) showed that
\begin{equation}
\label{eq:Landau-lim}
\limsup_{n\to\infty}\,\frac{n}{\varphi(n)\log\log n}=e^\gamma,
\end{equation}
where $\gamma$ is the \emph{Euler-Macheroni constant}. Eighty
years later, in a highly interesting work, Nicolas~\cite{Nic}
proved that the inequality
$$
\frac{n}{\varphi(n)}>e^\gamma\log\log n
$$
holds for infinitely many natural numbers $n$. Moreover, if $N_k$
denotes the product of the first $k$ primes, he proved that
$$
\frac{N_k}{\varphi(N_k)}>e^\gamma\log\log N_k
$$
holds for every $k\ge 1$ on the \emph{Riemann hypothesis}~(RH).
Assuming RH is false, he also showed there are both infinitely
many $k$ for which this inequality holds and infinitely many $k$
for which it does not hold. To acknowledge the many contributions
of Nicolas to this subject, we denote by $\cN$ the set of numbers
$n\in\N$ that satisfy the \emph{Nicolas inequality}:
\begin{equation}
\label{eq:Nicolas-ineq} \frac{n}{\varphi(n)}<e^\gamma\log\log n.
\end{equation}
The principle aim of this paper is to exhibit a broad class of
infinite subsets $\cS\subset\N$ such that this inequality holds
for all but finitely many $n\in\cS$. This class includes a set
that contains all natural numbers which can be expressed as a sum
of two squares.

Let $\sigma(n)$ be the \emph{sum of divisors function}. The
analogue of~\eqref{eq:Landau-lim} for this function was obtained
by Gronwall~\cite{Gron}, who proved that
$$
\limsup_{n\to\infty}\,\frac{\sigma(n)}{n\log\log n}=e^\gamma.
$$
Robin~\cite{Robin} showed that if RH is true, then the \emph{Robin
inequality}:
\begin{equation}
\label{eq:Robin-ineq} \frac{\sigma(n)}{n}<e^\gamma\log\log n
\end{equation}
holds for every integer $n>5040$, whereas if RH is false, then
this inequality fails for infinitely many $n$. We denote by $\cR$
the set of numbers $n\in\N$ that satisfy~\eqref{eq:Robin-ineq}. In
view of the elementary inequality
$$
\frac{\sigma(n)}{n}<\frac{n}{\varphi(n)}\qquad(n>1),
$$
it is clear that $\cN\subset\cR$.  Thus, for the class of subsets
$\cS\subset\N$ considered in the present paper, the Robin
inequality holds for all but finitely many $n\in\cS$.

Our work was originally inspired by a recent paper of Choie
\emph{et al}~\cite{CLMS}, which establishes the inclusion in $\cR$
of various infinite subsets of the natural numbers $\N$. In
particular, in~\cite{CLMS} it is shown that $\cR$ contains every
square-free number $n>30$, every odd integer $n>9$, every powerful
number $n>36$, and every integer $n>1$ not divisible by the fifth
power of some prime. As a consequence it follows that the RH holds
iff the Robin inequality holds for all natural numbers $n$
divisible by the fifth power of some prime. Note that this
criterion does not have the restriction $n\ge 5041$. Another
``5041-free'' criterion was given earlier by Lagarias~\cite{Lag},
who showed that RH is true iff
$$
\sigma(n)\le H_n+e^{H_n}\log H_n,
$$
where
$$
H_n=\sum_{j\le n}{1\over j}\qquad(n\ge 1).
$$

To state our results more precisely, let $\P$ denote the set of
prime numbers, and for any subset $\cA\subset\P$, put
$$
\pi_\cA(x)=\#\big\{p\le x~:~p\in\cA\big\}
$$
Let $\cP$ be an arbitrary (fixed) subset of $\P$ such that
\begin{equation}
\label{eq:deldel1}
\ud=\limsup_{x\to\infty}\frac{\pi_\cP(x)}{\pi(x)}<1 \mand
\ld=\liminf_{x\to\infty}\frac{\pi_\cP(x)}{\pi(x)}>0,
\end{equation}
where $\pi(x)=\#\{p\le x\}$ as usual. Let $\cQ$ denote the
complementary set of primes (i.e., $\cQ=\P\setminus\cP$), and note
that
\begin{equation}
\label{eq:deldel2}
\limsup_{x\to\infty}\frac{\pi_{\cQ}(x)}{\pi(x)}=1-\ld<1\mand
\liminf_{x\to\infty}\frac{\pi_{\cQ}(x)}{\pi(x)}=1-\ud>0.
\end{equation}
In this paper, we work with the set $\cS=\cS(\cP)$ defined by
\begin{equation}
\label{eq:S-defn}
\cS=\big\{n\in\N~:~\text{if~}p\in\cQ\text{~and~}p\mid
n,\text{~then~}p^{2}\mid n\big\}.
\end{equation}
Our main result is the following:

\begin{theorem}
\label{thm:main} The set $\cN$ contains all but finitely many of
the numbers in $\cS$.
\end{theorem}

\begin{corollary}
\label{cor:P-appl} Of the numbers $n$ which do not satisfy the
Nicolas inequality, all but finitely many are divisible by a prime
$q\in\cQ$ such that $q^2\nmid n$.
\end{corollary}

In particular, for any fixed $a,m\in\N$ with $\gcd(a,m)=1$, one
can put
$$
\cP=\big\{p\in\P~:~p\not\equiv a\pmod m\big\}
$$
and apply Corollary~\ref{cor:P-appl} to deduce the following:

\begin{corollary}
\label{cor:P-appl2} Of the numbers $n$ which do not satisfy the
Nicolas inequality, all but finitely many are divisible by a prime
$q\equiv a\pmod m$ such that $q^2\nmid n$.
\end{corollary}

In Section~\ref{sec:explicit} we examine more closely the special
case that
$$
\cP=\big\{p\in\P~:~p\equiv 1\pmod 4\big\}\cup\{2\}.
$$
Note that the corresponding set $\cS$ contains all natural numbers
of the form $n=a^2+b^2$ (since, by a theorem of Fermat, every
prime $q\equiv 3\pmod 4$ appears with even multiplicity in the
prime factorization of $n$ if and only if $n$ can be written as a
sum of two squares). Using effective bounds from~\cite{RamRum} on
the number of primes in arithmetic progressions modulo~$4$, we are
able to determine the set $\cS\setminus\cN$ completely, leading
to:

\begin{theorem}
\label{thm:main2} The set $\cS\setminus\cN$ contains precisely
$347$ natural numbers.  In particular, there are precisely $246$
numbers which can be expressed as a sum of two squares and such
that the Nicolas inequality~\eqref{eq:Nicolas-ineq} does not hold,
the largest of which is the number $52509581344222812810$.
\end{theorem}

As an application, we obtain the unconditional result that
$$
\{1, 2, 4, 5, 8, 9, 10, 16, 18, 20, 36, 72, 180, 360, 720\}
$$
is a complete list of those natural numbers which can be expressed
as a sum of two squares and such that the Robin
inequality~\eqref{eq:Robin-ineq} does not hold; this result is
consistent with the truth of the Riemann Hypothesis.

Results like those of Theorem~\ref{thm:main2} can be established
for certain quadratic forms other than $a^2+b^2$. For example,
using similar techniques one finds that there are precisely $261$
numbers that can be expressed in the form $n=a^2+3b^2$ and for
which the Nicolas inequality~\eqref{eq:Nicolas-ineq} does not
hold, the largest of which is the number $397999936131188090700$.

\bigskip

Throughout the paper, any implied constants in the symbols $O$,
$\ll$, $\gg$ and $\asymp$ depend (at most) on the set $\cP$ and
are absolute otherwise.  We recall that for positive functions
$f,g$ the notations $f=O(g)$, $f\ll g$ and $g\gg f$ are all
equivalent to the assertion that $f\le c g$ for some constant
$c>0$, and the notation $f\asymp g$ means that $f\ll g$ and $g\ll
f$.

\section{Proof of Theorem~\ref{thm:main}}
For every natural number $n$ we put
$$
F(n)=\frac{n}{\varphi(n)}=\prod_{p\,\mid\,n}\frac{p}{p-1}\,.
$$
Note that
\begin{equation}
\label{eq:kernel-stuff} F(n)=F(\kappa(n))\mand
\omega(n)=\omega(\kappa(n)),
\end{equation}
where $\omega(n)$ is the number of distinct prime divisors of $n$,
and $\kappa(n)$ is the square-free kernel of $n$:
$$
\kappa(n)=\prod_{p\,\mid\,n}p.
$$
Let
$$
\cNnot=\N\setminus\cN=\big\{n\in\N~:~F(n)\ge e^\gamma \log\log
n\big\},
$$
and for every integer $k\ge 0$, let
$$
\cV_k=\big\{n\in\N~:~\omega(n)\ge k\big\}\mand
\cW_k=\cS\cap\cNnot\cap\cV_k.
$$
Since $\cV_0=\N$, Theorem~\ref{thm:main} is the assertion that
$\cW_0=\cS\cap\cNnot$ is a finite set. In view of the next lemma,
it suffices to show that $\cW_k=\varnothing$ for some $k$.

\begin{lemma}
\label{lem:W0/Wk} For every $k\ge 0$, $\cW_0\setminus\cW_k$ is a
finite set.
\end{lemma}

Since $\omega(n)<k$ and $F(n)\ge e^\gamma \log\log n$ for all
$n\in\cW_0\setminus\cW_k$, Lemma~\ref{lem:W0/Wk} is an immediate
consequence of the following:

\begin{lemma}
\label{lem:K} For every constant $K>0$, there are at most finitely
many natural numbers $n$ such that $\omega(n)\le K$ and $F(n)\ge
e^\gamma \log\log n$.
\end{lemma}

\begin{proof}
If $\overline p_1,\overline p_2,\ldots$ is the sequence of
consecutive prime numbers, then for any such number $n$ we have
$$
\prod_{j\le K}\frac{\overline p_j}{\overline
p_j-1}\ge\prod_{p\,\mid\,n}\frac{p}{p-1}
 =F(n)\ge e^\gamma\log\log n;
$$
this shows that $n$ is bounded by a constant which depends only on
$K$.
\end{proof}

For every natural number $n$, let
$$
s(n)=\biggl(\,\prod_{\substack{p\,\mid\,n\\p\in\cP}}p\biggl)
\biggl(\,\prod_{\substack{q\,\mid n\\q\in\cQ}}q^2\biggl),
$$
and put
$$
\cY=\big\{n\in\N~:~n=s(n)\big\}.
$$
Note that $\cY\subset\cS$. The following statements are
elementary:
\begin{itemize}
\item[$(\sC_1)$]\quad if $n=pm$ with $p\in\cP$ and $p\nmid m$,
then $n\in\cY$ if and only if $m\in\cY$;

\item[$(\sC_2)$]\quad if $n=q^2m$ with $q\in\cQ$ and $q\nmid m$,
then $n\in\cY$ if and only if $m\in\cY$;

\item[$(\sC_3)$]\quad $s(n)\in\cS$ for all $n$;

\item[$(\sC_4)$]\quad $\kappa(s(n))=\kappa(n)$ for all $n$;

\item[$(\sC_5)$]\quad $s(n)\mid n$ for all $n\in\cS$; in
particular, $s(n)\le n$.
\end{itemize}

\begin{lemma}
\label{lem:s(n)} If $\cW_k\ne\varnothing$ and $m_k$ is the least
integer in $\cW_k$, then $m_k\in\cY$.
\end{lemma}

\begin{proof}
Clearly, $s(m_k)\in\cS$ by $(\sC_3)$. Combining $(\sC_4)$
with~\eqref{eq:kernel-stuff} one sees that
$$
F(s(n))=F(n)\mand \omega(s(n))=\omega(n)\qquad(n\in\N).
$$
Then, using $(\sC_5)$ it follows that
$$
F(s(m_k))=F(m_k)\ge e^\gamma\log\log m_k\ge e^\gamma\log\log
s(m_k),
$$
which shows that $s(m_k)\in\cNnot$. Finally, $s(m_k)\in\cV_k$
since
$$
\omega(s(m_k))=\omega(m_k)\ge k.
$$
Thus, we have shown that $s(m_k)\in\cS\cap\cNnot\cap\cV_k=\cW_k$.
Since $m_k$ is the \emph{least} integer in $\cW_k$, the equality
$m_k=s(m_k)$ follows from $(\sC_5)$, hence $m_k\in\cY$.
\end{proof}

Next, for every integer $k\ge 0$ let
$$
\cZ_k=\big\{n\in\N~:~\Omega(n)=k\big\}\mand
\cT_k=\cNnot\cap\cY\cap\cZ_k.
$$
Here, $\Omega(n)$ is the number of prime divisors of $n$, counted
with multiplicity. Using Lemma~\ref{lem:s(n)} one sees that if
$\cW_\ell\ne\varnothing$ and $m_\ell$ is the least integer in
$\cW_\ell$, then $m_\ell\in\cT_k$ for some $k\ge\ell$; in
particular,
$$
\bigcup_{k\ge\ell}\cT_k=\varnothing\quad
\Longrightarrow\quad\cW_\ell=\varnothing.
$$
As we mentioned earlier, in order to prove Theorem~\ref{thm:main}
it suffices to show that $\cW_\ell=\varnothing$ for some $\ell$,
hence it is enough to show that $\cT_k\ne\varnothing$ for at most
finitely many integers $k\ge 0$.

When $\cT_k\ne\varnothing$ we shall use the following notation.
Let $n_k$ denote the least integer in $\cT_k$. Let $\pk$ be the
largest prime $p\in\cP$ that divides $n_k$, and put $\pk=1$ if no
such prime exists. Similarly, let $\qk$ be the largest prime
$q\in\cQ$ that divides $n_k$, and set $\qk=1$ if no such prime
exists. Finally, let
\begin{equation}
\label{eq:defn} P_k^+=\max\{\pk,\qk\}\mand P_k^-=\min\{\pk,\qk\}.
\end{equation}
Note that $P_k^+$ is the largest prime factor of $n_k$.

\newpage

\begin{lemma}
\label{lem:trick1} Suppose $\cT_k\ne\varnothing$:
\begin{itemize}
\item[$(i)$] if $p\in\cP$ with $p<\pk$, then $p\mid n_k$;

\item[$(ii)$] if $q\in\cQ$ with $q<\qk$, then $q\mid n_k$.
\end{itemize}
\end{lemma}

\begin{proof}
Suppose on the contrary that $p\in\cP$ with $p<\pk$ and $p\nmid
n_k$. Since $n_k=s(n_k)$ we can write $n_k=\pk m$ with $\pk\nmid
m$. Put $n^*=pm$. Since $n_k\in\cNnot$, $F(p)>F(\pk)$, and
$n^*<n_k$, it follows that
$$
F(n^*)=F(p)\,F(m)>F(\pk)\,F(m)=F(n_k)\ge e^\gamma\log\log
n_k>e^\gamma\log\log n^*,
$$
where we have used the fact that $F$ is multiplicative; this shows
that $n^*\in\cNnot$. As $n_k\in\cY$, $(\sC_1)$ implies that
$n^*\in\cY$. Finally, since $\Omega$ is (completely) additive, we
see that
$$
\Omega(n^*)=\Omega(m)+1=\Omega(n_k)=k,
$$
which shows that $n^*\in\cZ_k$, and thus
$n^*\in\cNnot\cap\cY\cap\cZ_k=\cT_k$. But this is impossible since
$n^*<n_k$ (the least number in $\cT_k$), and this contradiction
completes our proof of~$(i)$.  Using $(\sC_2)$, the proof
of~$(ii)$ is similar; we omit the details.
\end{proof}

\begin{lemma}
\label{lem:trick2} Suppose that $\cT_k\ne\varnothing$ and
$\pk<\qk$. Then there is at most one prime $p\in\cP$ such that
$\pk<p<\qk$.
\end{lemma}

\begin{proof}
Suppose on the contrary that there are two primes $p_1,p_2\in\cP$
such that $\pk<p_1<p_2<\qk$. Since $n_k=s(n_k)$ we can write
$n_k=\qk^{\,\,2}m$, and it is clear that $\gcd(m,p_1p_2\qk)=1$.
Put $n^*=p_1p_2m$. Since $n_k\in\cNnot$,
$F(p_1p_2)>F(\qk^{\,\,2})$, and $n^*<n_k$, we have
$$
F(n^*)=F(p_1p_2)\,F(m)>F(\qk^{\,\,2})\,F(m)=F(n_k)\ge
e^\gamma\log\log n_k>e^\gamma\log\log n^*,
$$
which shows that $n^*\in\cNnot$. As $n_k\in\cY$, $(\sC_1)$ implies
that $n^*\in\cY$. Finally, since
$$
\Omega(n^*)=\Omega(m)+2=\Omega(n_k)=k,
$$
we see that $n^*\in\cZ_k$, and thus
$n^*\in\cNnot\cap\cY\cap\cZ_k=\cT_k$. But this is impossible since
$n^*<n_k$, and this contradiction implies the result.
\end{proof}

\begin{lemma}
\label{lem:trick3} Suppose that $\cT_k\ne\varnothing$ and
$\pk>\qk$. Let $p$ be the largest prime in $\cP$ that is less than
$\pk$, and let $q$ be the smallest prime in $\cQ$ that is greater
than~$\qk$. Then $q>p/2$.
\end{lemma}

\begin{proof}
Suppose on the contrary that $q\le p/2$. Since $n_k=s(n_k)$ and
$p\mid n_k$ (by Lemma~\ref{lem:trick1}) but $q\nmid n_k$ (since
$q>\qk$), we can write $n_k=p\pk m$, where $\gcd(m,p\pk q)=1$. Put
$n^*=q^2m$. As in the proofs of Lemmas~\ref{lem:trick1}
and~\ref{lem:trick2}, we see that $n^*\in\cY\cap\cZ_k$. Since
$p<\pk$ and $q\le p/2$, we have
$$
F(p\pk)=\frac{p\pk}{(p-1)(\pk-1)}<\frac{p^2}{(p-1)^2}
<\frac{q}{q-1}=F(q^2);
$$
therefore,
$$
F(n^*)=F(q^2)\,F(m)>F(p\pk)\,F(m)=F(n_k)\ge e^\gamma\log\log
n_k>e^\gamma\log\log n^*,
$$
which shows that $n^*\in\cNnot$.  Thus,
$n^*\in\cNnot\cap\cY\cap\cZ_k=\cT_k$. But this is impossible since
$n^*<n_k$, and this contradiction implies the result.
\end{proof}

As mentioned above, in order to prove Theorem~\ref{thm:main} it
suffices to show that $\cT_k\ne\varnothing$ for at most finitely
many integers $k\ge 0$. Arguing by contradiction, we shall assume
that the set
$$
\cK=\{k\ge 0~:~\cT_k\ne\varnothing\}
$$
has infinitely many elements.

Since $\Omega(n_k)=k$, we see that $n_k\to\infty$ as $k\to\infty$
with $k\in\cK$; using Lemma~\ref{lem:K} it follows that
$\omega(n_k)\to\infty$ as well, and therefore $P_k^+\to\infty$.

We claim that
\begin{equation}
\label{eq:pkqk} \pk\asymp \qk\qquad(k\in\cK),
\end{equation}
which by~\eqref{eq:defn} is equivalent to
\begin{equation}
\label{eq:P+P-} P_k^+\asymp P_k^-\qquad(k\in\cK).
\end{equation}
To see this, we express $\cK$ as a disjoint union $\cA\cup\cB$,
where $\cA$ [resp.~$\cB$] is the set of numbers $k\in\cK$ for
which $\pk<\qk$ [resp.~$\pk>\qk$].  To prove~\eqref{eq:pkqk} it
suffices to show:
\begin{itemize}

\item[$(\sD_1)$]\quad $\pk\gg\qk$ for all $k\in\cA$;

\item[$(\sD_2)$]\quad $\pk\ll\qk$ for all $k\in\cB$.
\end{itemize}
We use the following result, which is an easy consequence of the
prime number theorem:

\begin{lemma}
\label{lem:pnt-ap} Let $c_{\cP}=\ud/\ld$ and
$c_{\cQ}=\(1-\ld\,\)/\(1-\ud\,\)$. For every $\eps>0$ there is a
number $x_0(\eps)$ such that for all $x>x_0(\eps)$:
\begin{enumerate}
\item[$(i)$] if $p$ is the smallest prime in $\cP$ greater than
$x$, then $p\le \(c_{\cP}+\eps\)x$;

\item[$(ii)$] if $q$ is the smallest prime in $\cQ$ greater than
$x$, then $q\le\(c_{\cQ}+\eps\)x$;

\item[$(iii)$] if $p$ is the largest prime in $\cP$ less than $x$,
then $p\ge\(c^{-1}_{\cP}-\eps\)x$;

\item[$(iv)$] if $q$ is the largest prime in $\cQ$ less than $x$,
then $q\ge\(c^{-1}_{\cQ}-\eps\)x$.
\end{enumerate}
\end{lemma}

To prove $(\sD_1)$ we can assume that $\cA$ is an infinite set.
Let $k\in\cA$, so that $\pk<\qk$. Since $\qk=P_k^+\to\infty$ as
$k\to\infty$ with $k\in\cA$, the assertion $(\sD_1)$ then follows
from Lemmas~\ref{lem:trick2} and~\ref{lem:pnt-ap}.

To prove $(\sD_2)$ we can assume that $\cB$ is an infinite set.
Let $k\in\cB$, so that $\pk>\qk$. Let $p,q$ be defined as in
Lemma~\ref{lem:trick3}. Since $\pk=P_k^+\to\infty$ as $k\to\infty$
with $k\in\cB$, on combining Lemmas~\ref{lem:trick3}
and~\ref{lem:pnt-ap} it follows that
$$
\pk\ll p\ll q\ll\qk,
$$
which proves $(\sD_2)$ and completes our proof of~\eqref{eq:pkqk}.

Next, for every $n\in\N$ let
$$
\omega_\cP(n)=\#\big\{p\in \cP~:~p\mid n\big\} \mand
\omega_\cQ(n)=\#\big\{q\in \cQ~:~q\mid n\big\}.
$$
We claim that
\begin{equation}
\label{eq:omega-asymp}
\omega_{\cP}(n_k)\asymp\omega_{\cQ}(n_k)\qquad(k\in\cK).
\end{equation}
Indeed, by Lemma~\ref{lem:trick1} it follows that
$\omega_{\cP}(n_k)=\pi_{\cP}(\pk)$ and
$\omega_{\cQ}(n_k)=\pi_{\cQ}(\qk)$. Therefore, using the prime
number theorem together with~\eqref{eq:deldel1},
\eqref{eq:deldel2} and~\eqref{eq:pkqk} we have
$$
\omega_{\cP}(n_k)=\pi_{\cP}(\pk)\asymp\frac{\pk}{\log\pk}
\asymp\frac{\qk}{\log\qk}\asymp\pi_{\cQ}(\qk) =\omega_{\cQ}(n_k),
$$
which proves~\eqref{eq:omega-asymp}.

Finally, we need the following relation:
\begin{equation}
\label{eq:rad-omega}
\log\kappa(n_k)\asymp\omega(n_k)\log\omega(n_k)\qquad(k\in\cK).
\end{equation}
To prove this, observe that the definition~\eqref{eq:defn} and
Lemma~\ref{lem:trick1} together imply
$$
\prod_{p\le P_k^-}p~\biggl|~\kappa(n_k)\mand
\kappa(n_k)~\biggl|~\prod_{p\le P_k^+}p.
$$
Consequently,
$$
\sum_{p\le P_k^-}\log p\le\log\kappa(n_k)\le
 \sum_{p\le P_k^+}\log p,
$$
and also
$$
\pi(P_k^-)\le\omega(n_k)\le\pi(P_k^+).
$$
By the prime number theorem, for either choice of the sign $\pm$
we have
$$
\sum_{p\le P_k^\pm}\log p\sim
P_k^\pm\mand\pi(P_k^\pm)\sim\frac{P_k^\pm}{\log
P_k^\pm}\qquad(k\to\infty,~k\in\cK),
$$
therefore in view of~\eqref{eq:P+P-} we see that
$$
\log\kappa(n_k)\asymp P_k^+\mand\omega(n_k)\asymp\frac{P_k^+}{\log
P_k^+}\,,
$$
and~\eqref{eq:rad-omega} follows immediately.

Now we come to the heart of the argument. To complete the proof of
Theorem~\ref{thm:main}, we seek a contradiction to our assumption
that $\cK$ is an infinite set.  For this, it is enough to prove
both of the following statements with a suitably chosen real
number $\eps>0$:
\begin{itemize}

\item[$(\sE_1)$] the inequality $n_k\le\kappa(n_k)^{1+\eps}$ holds
for at most finitely many $k\in\cK$;

\item[$(\sE_2)$] the inequality $n_k>\kappa(n_k)^{1+\eps}$ holds
for at most finitely many $k\in\cK$.
\end{itemize}
In view of~\eqref{eq:omega-asymp} and~\eqref{eq:rad-omega}, there
is a constant $C>1$ such that the inequalities
\begin{equation}
\label{eq:omega-asymp-expl} \omega_{\cP}(n_k)\le
(C-1)\,\omega_{\cQ}(n_k)
\end{equation}
and
\begin{equation}
\label{eq:rad-omega-expl} \log\kappa(n_k)\le
C\,\omega(n_k)\log\omega(n_k)
\end{equation}
both hold if $k$ is sufficiently large.  Let $C$ be fixed, and put
$\eps=C^{-3}$.

To prove $(\sE_1)$, we suppose on the contrary that
$n_k\le\kappa(n_k)^{1+\eps}$ holds for infinitely many $k\in\cK$.
Let $k$ be large, and put
$$
r=\omega_\cP(n_k)=\pi_\cP(\pk) \mand
s=\omega_\cQ(n_k)=\pi_\cQ(\qk)
$$
By what we have already seen it is clear that
$\min\{r,s\}\to\infty$ as $k\to\infty$ with $k\in\cK$, thus
by~\eqref{eq:omega-asymp-expl} we have
\begin{equation}
\label{eq:srboundit} r\le (C-1)s
\end{equation}
if $k$ is large enough.  By Lemma~\ref{lem:trick1} and the fact
that $n_k\in\cY$, it follows that
$$
n_k=\biggl(\,\prod_{\substack{p\le\pk\\p\in\cP}}p\biggl)
\biggl(\,\prod_{\substack{q\le\qk\\q\in\cQ}}q^2\biggl)\mand
\kappa(n_k)=\biggl(\,\prod_{\substack{p\le\pk\\p\in\cP}}p\biggl)
\biggl(\,\prod_{\substack{q\le\qk\\q\in\cQ}}q\biggl).
$$
Hence, our assumption that $n_k\le\kappa(n_k)^{1+\eps}$ implies
that
\begin{equation}
\label{eq:knkbound}
\kappa(n_k)\ge\(\frac{n_k}{\kappa(n_k)}\)^{1/\eps}=
\biggl(\,\prod_{\substack{q\le\qk\\q\in\cQ}}q\biggl)^{1/\eps}.
\end{equation}
If $\overline p_1,\overline p_2,\ldots$ is the sequence of
consecutive prime numbers, then by the prime number theorem (and
recalling our choice of $\eps$) we derive that
$$
\log\kappa(n_k)\ge C^3\sum_{\substack{q\le\qk\\q\in\cQ}}\log q\ge
C^3\sum_{p\le \overline p_s}\log p\sim C^3\overline p_s\sim
C^3s\log s
$$
as $k\to\infty$ with $k\in\cK$. On the other hand,
using~\eqref{eq:rad-omega-expl}, \eqref{eq:srboundit} and the fact
that $\omega(n_k)=r+s$, it follows that
$$
\log\kappa(n_k)\le C(r+s)\log(r+s)\le C^2s\log(Cs)\sim C^2s\log s.
$$
Since $C^3>C^2$, these two inequalities for $\log\kappa(n_k)$ lead
to a contradiction once $k$ is sufficiently large, and this
completes the proof of $(\sE_1)$.

To prove $(\sE_2)$ we use some ideas from Choie \emph{et
al}~\cite{CLMS}. Suppose that $n_k>\kappa(n_k)^{1+\eps}$, and put
$t=\omega(n_k)$. We claim that either
\begin{equation}
\label{eq:bd1} \sum_{p\le\overline p_t}\log
p<(1+\eps)^{-1/2}\,\overline p_t,
\end{equation}
or
\begin{equation}
\label{eq:bd2} \overline p_t\le\exp\big(2/\log(1+\eps)\big).
\end{equation}
Assuming the claim, it is easy to see that $\omega(n_k)$ is
bounded above by a constant $K$ that depends only on $\eps$.  By
Lemma~\ref{lem:K}, $n_k$ can take only finitely many distinct
values, which implies $(\sE_2)$.

To prove the claim, assume that~\eqref{eq:bd1} fails:
$$
\log(\overline p_1\cdots\overline p_t)=\sum_{p\le\overline
p_t}\log p\ge(1+\eps)^{-1/2}\,\overline p_t.
$$
Thanks to Rosser and Schoenfeld~\cite{RosSch} it is known that
$$
\prod_{p\le x}\frac{p}{p-1}\le e^\gamma\(\log x+\frac{1}{\log
x}\)\qquad(x>1).
$$
Therefore, taking $x=\overline p_t$ and noting that
$\kappa(n_k)\ge\overline p_1\cdots\overline p_t$, we derive that
\begin{equation*}
\begin{split}
e^\gamma\(\log \overline p_t+\frac{1}{\log \overline p_t}\)&\ge
\prod_{j=1}^{t}\frac{\overline p_j}{\overline p_j-1}\ge
\frac{n_k}{\varphi(n_k)}\ge e^\gamma\log\log n_k\\
&> e^\gamma\log\((1+\eps)\log\kappa(n_k)\)\\
&\ge e^\gamma\log\((1+\eps)\log(\overline p_1\cdots\overline
p_t)\)\\
&\ge e^\gamma\log\((1+\eps)^{1/2}\,\overline
p_t\)=e^\gamma\(\log\overline p_t+0.5\log(1+\eps)\);
\end{split}
\end{equation*}
that is,
$$
\frac{1}{\log \overline p_t}\ge 0.5\log(1+\eps),
$$
which is equivalent to~\eqref{eq:bd2}.  This proves the claim and
completes our proof of Theorem~\ref{thm:main}.

\section{Proof of Theorem~\ref{thm:main2}}
\label{sec:explicit}

We continue to use the notation of the previous section, but we
focus on the special case that
\begin{equation*}
\begin{split}
\cP&=\big\{p\in\P~:~p\equiv 1\pmod 4\big\}\cup\{2\},\\
\cQ&=\big\{q\in\P~:~q\equiv 3\pmod 4\big\}.
\end{split}
\end{equation*}
Note that the corresponding set $\cS$ contains every natural
number that can be expressed as a sum of two squares. As before,
we write
$$
\cT_k=\big\{n\in\N~:~F(n)\ge e^\gamma \log\log n,~n=s(n),
\text{~and~}\Omega(n)=k\big\}
$$
and put
$$
\cK=\{k\ge 0~:~\cT_k\ne\varnothing\}.
$$

\begin{lemma}
\label{lem:newtrick1} If $k\in\cK$, then $P_k^-<50000$.
\end{lemma}

\begin{proof}
For every real number $x\ge 10$, let
\begin{itemize}
\item $g_\cP(x)=$ the smallest prime in $\cP$ greater than $x$;

\item $g_\cQ(x)=$ the smallest prime in $\cQ$ greater than $x$;

\item $\ell_\cP(x)=$ the largest prime in $\cP$ less than $x$;

\item $\ell_\cQ(x)=$ the largest prime in $\cQ$ less than $x$.
\end{itemize}
Also, put
$$
\vartheta_\cP(x)=\sum_{\substack{p\le x\\p\in\cP}}\log p\mand
\vartheta_\cQ(x)=\sum_{\substack{q\le x\\q\in\cQ}}\log q.
$$
Using the explicit bounds of Theorems~1 and~2 of Ramar\'e and
Rumely~\cite{RamRum}, we see that the inequalities
\begin{equation}
\label{eq:theta-ests} 0.49\,x<\vartheta_\cP(x)<0.51\,x\mand
0.49\,x<\vartheta_\cQ(x)<0.51\,x.
\end{equation}
hold for all $x\ge 45000$ (note that
$\vartheta_\cP(x)=\log2+\theta(x;4,1)$ and
$\vartheta_\cQ(x)=\theta(x;4,3)$ in the notation
of~\cite{RamRum}). Consequently, for any $x\ge 50000$ we have
$$
\tfrac{49}{51}\,x<\ell_\cP(x)<x<g_\cP(x)<\tfrac{51}{49}\,x
$$
and
$$
\tfrac{49}{51}\,x<\ell_\cQ(x)<x<g_\cQ(x)<\tfrac{51}{49}\,x.
$$

Now suppose that $P_k^-\ge 50000$. Using Lemma~\ref{lem:trick2}
and the preceding bounds we have
$$
\qk<g_\cP(g_\cP(\pk))<\(\tfrac{51}{49}\)^2\pk.
$$
On the other hand, by Lemma~\ref{lem:trick3} we have
$$
\tfrac{51}{49}\,\qk>g_\cQ(\qk)>\tfrac12\,\ell_\cP(\pk)
>\tfrac{49}{102}\,\pk.
$$
Hence, it follows that
\begin{equation}
\label{eq:superman} 0.92\,\qk<\pk<2.2\,\qk.
\end{equation}
By Lemma~\ref{lem:trick1} it is clear that
$$
\log\kappa(n_k)=\sum_{\substack{p\le\pk\\p\in\cP}}\log
p+\sum_{\substack{q\le\qk\\q\in\cQ}}\log
q=\vartheta_\cP(\pk)+\vartheta_\cQ(\qk).
$$
On the other hand, arguing as in the proof of
Theorem~\ref{thm:main}, it follows from~\eqref{eq:knkbound} that
$$
\log\kappa(n_k)\ge\eps^{-1}\vartheta_\cQ(\qk)
$$
if $\eps>0$ is fixed and $n_k\le\kappa(n_k)^{1+\eps}$. Combining
the two preceding results with~\eqref{eq:theta-ests}, we see that
$$
0.51\,(\pk+\qk)\ge\vartheta_\cP(\pk)+\vartheta_\cQ(\qk)
\ge\eps^{-1}\vartheta_\cQ(\qk)\ge 0.49\,\eps^{-1}\qk
$$
since $P_k^-\ge 50000$; taking into account~\eqref{eq:superman},
we further have
$$
0.51\,(1+2.2)\,\qk\ge 0.51\,(\pk+\qk)\ge 0.49\,\eps^{-1}\qk,
$$
which implies that $\eps\ge 0.3002$.  Thus, for the smaller value
$\eps=0.3$, we see that the condition $n_k\le\kappa(n_k)^{1.3}$
implies $P_k^-<50000$.

On the other hand, if $n_k>\kappa(n_k)^{1.3}$, we put
$t=\omega(n_k)$ as in the proof of Theorem~\ref{thm:main}.  Since
$\eps=0.3$, we derive from~\eqref{eq:bd1} and~\eqref{eq:bd2} that
either
\begin{equation}
\label{eq:bd3} \vartheta(\overline p_t)=\sum_{p\le\overline
p_t}\log p<(1.3)^{-1/2}\,\overline p_t<0.88\,\overline p_t,
\end{equation}
or
$$
\overline p_t\le\exp(2/\log 1.3)<2045.
$$
Using again Theorems~1 and~2 of Ramar\'e and Rumely~\cite{RamRum}
(see also~\cite{RosSch}), it is easy to see that the
inequality~\eqref{eq:bd3} implies $\overline p_t<300$, hence the
inequality $\overline p_t<2045$ holds in both cases.  It follows
that $t<310$, and therefore,
$$
\min\{\pi_{\cP}(\pk),\pi_{\cQ}(\qk)\}
=\min\{\omega_\cP(n_k),\omega_\cQ(n_k)\}\le\omega(n_k)=t<310,
$$
which implies that $P_k^-<5000$.  This completes the proof.
\end{proof}

\begin{corollary}
\label{cor:Kbound} If $k\in\cK$, then $k<10000$.
\end{corollary}

\begin{proof}
For any $k\in\cK$ we have
$$
k=\Omega(n_k)=\omega_\cP(n_k)+2\,\omega_\cQ(n_k)
=\pi_\cP(\pk)+2\,\pi_\cQ(\qk).
$$
If $P_k^-=\pk$ (i.e., $\pk<\qk$), then by Lemmas~\ref{lem:trick2}
and~\ref{lem:newtrick1} it follows that
\begin{equation*}
\begin{split}
k&\le\max_{p<50000}\left\{\pi_\cP(p)
+2\,\pi_\cQ\big(g_\cP(g_\cP(p))\big)\right\}\\
&\le\pi_\cP(50000)+2\,\pi_\cQ\big(g_\cP(g_\cP(50000))\big)=7718.
\end{split}
\end{equation*}
If $P_k^-=\qk$ (i.e., $\qk<\pk$), then by Lemmas~\ref{lem:trick3}
and~\ref{lem:newtrick1} it follows that
\begin{equation*}
\begin{split}
k&\le\max_{q<50000}\max_{\substack{p\in\P\\\ell_\cP(p)<2g_\cQ(q)}}
\left\{\pi_\cP(p)+2\,\pi_\cQ(q)\right\}\\
&=\max_{q<50000}\max_{\substack{p\in\P\\\ell_\cP(p)<2g_\cQ(q)}}
\left\{1+\pi_\cP(\ell_\cP(p))+2\,\pi_\cQ(q)\right\}\\
&\le\max_{q<50000}\left\{1+\pi_\cP(2\,g_\cQ(q))+2\,\pi_\cQ(q)\right\}\\
&\le 1+\pi_\cP(2\,g_\cQ(50000))+2\,\pi_\cQ(50000)=9951.
\end{split}
\end{equation*}
The result follows.
\end{proof}

Now let $\oop_1,\oop_2,\ldots$ be the sequence of consecutive
primes in $\cP$, and let $\ooq_1,\ooq_2,\ldots$ be the consecutive
primes in $\cQ$. For any integers $r,s\ge 0$, let
$$
N_{r,s}=\biggl(\,\prod_{i=1}^r\oop_i\biggl)
\biggl(\,\prod_{j=1}^s\ooq\vphantom q_j^{\,2}\biggl).
$$
It is easy to see that $N_{r,s}\in\cY$ for all $r,s\ge 0$, and for
every $k\in\cK$ one has
$$
n_k=N_{r,s},\qquad \pk=\oop_r,\qquad \qk=\ooq_s\mand k=r+2s,
$$
where $r=\omega_\cP(n_k)$ and $s=\omega_\cQ(n_k)$. By a
straightforward computation, one verifies the following:
\begin{lemma}
\label{lem:Nrs} If $r,s\ge 0$, then $N_{r,s}\in\cNnot$ if and only
if the pair $(r,s)$ lies in the set
\begin{equation*}
\begin{split}
\cX=\big\{&(0,0),(1,0),(0,1),(2,0),(1,1),(2,1),(1,2),(3,1),(2,2),(4,1),\\
&(3,2),(2,3),(4,2),(3,3),(5,2),(4,3),(3,4),(5,3),(4,4),(6,3),\\
&(5,4),(4,5),(7,3),(6,4),(5,5),(7,4),(6,5),(7,5),(8,5)\big\}.
\end{split}
\end{equation*}
\end{lemma}

We remark that, in view of Corollary~\ref{cor:Kbound}, it suffices
to check the condition $N_{r,s}\in\cNnot$ only for those pairs
$(r,s)$ with $r+2s<10000$.

\begin{corollary}
If $k\in\cK$, then $k\le 18$.
\end{corollary}

\begin{corollary}
If $n\in\cS\cap\cNnot$, $r=\omega_\cP(n)$ and $s=\omega_\cQ(n)$,
then $(r,s)\in\cX$. In particular, $\omega(n)\le 13$.
\end{corollary}

\begin{proof}
Since
$$
F(N_{r,s})=\biggl(\,\prod_{i=1}^r\frac{\oop_i}{\oop_i-1}\biggl)
\biggl(\,\prod_{j=1}^s\frac{\ooq_j}{\ooq_j-1}\biggl) \ge
\biggl(\,\prod_{\substack{p\,\mid\,n\\p\in\cP}}\frac{p}{p-1}\biggl)
\biggl(\,\prod_{\substack{q\,\mid\,n\\q\in\cQ}}\frac{q}{q-1}\biggl)=F(n)
$$
and
$$
n\ge s(n)=\biggl(\,\prod_{\substack{p\,\mid\,n\\p\in\cP}}p\biggl)
\biggl(\,\prod_{\substack{q\,\mid\,n\\q\in\cQ}}q^2\biggl)\ge
\biggl(\,\prod_{i=1}^r\oop_i\biggl)
\biggl(\,\prod_{j=1}^s\ooq\vphantom q_j^{\,2}\biggl)=N_{r,s},
$$
we have
$$
F(N_{r,s})\ge F(n)\ge e^\gamma\log\log n\ge e^\gamma\log\log
N_{r,s},
$$
which shows that $N_{r,s}\in\cNnot$.
\end{proof}

We now turn to a description of our method for generating the
elements of $\cS\setminus\cN=\cS\cap\cNnot$.  For any given
$n\in\cS\cap\cNnot$ with $r=\omega_\cP(n)$ and $s=\omega_\cQ(n)$,
we can write
$$
s(n)=p_1\cdots p_r\,q_1^2\cdots q_s^2,
$$
where $p_1<\cdots<p_r$ are primes in $\cP$ and $q_1<\cdots<q_s$
are primes in $\cQ$. For fixed $i=1,\ldots,r$, let $\gamma_i$ be
the largest non-negative integer such that the number
$$
\biggl(\,\prod_{\ell=1}^{i-1}\oop_\ell\biggl)
\biggl(\,\prod_{\ell=i}^r\oop_{\ell+\gamma_i}\biggl)
\biggl(\,\prod_{j=1}^s\ooq\vphantom q_j^{\,2}\biggl)
$$
lies in $\cNnot$, which exist by Lemma~\ref{lem:K}. Using an
argument similar to that in the proof of Lemma~\ref{lem:trick1},
one can deduce that
\begin{equation}
\label{eq:oo1} \oop_i\le
p_i\le\oop_{i+\gamma_i}\qquad(i=1,\ldots,r).
\end{equation}
Similarly, for fixed $j=1,\ldots,s$, let $\delta_j$ be the largest
non-negative integer such that the number
$$
\biggl(\,\prod_{i=1}^r\oop_i\biggl)
\biggl(\,\prod_{\ell=1}^{j-1}\ooq_\ell^{\,2}\biggl)
\biggl(\,\prod_{\ell=j}^s\ooq\vphantom
q_{\ell+\delta_j}^{\,2}\biggl)
$$
lies in $\cNnot$.  Then,
\begin{equation}
\label{eq:oo2} \ooq_j\le
q_j\le\ooq_{j+\gamma_j}\qquad(j=1,\ldots,s).
\end{equation}
Therefore, for fixed $(r,s)\in\cX$, if $n\in\cS\cap\cNnot$ with
$r=\omega_\cP(n)$ and $s=\omega_\cQ(n)$, then the number $s(n)$
must lie in the finite set $\cA_{r,s}$ of integers of the form
\begin{equation}
\label{eq:factorm} m=p_1\cdots p_r\,q_1^2\cdots q_s^2,
\end{equation}
where $p_1<\cdots<p_r$ are primes in $\cP$, $q_1<\cdots<q_s$ are
primes in $\cQ$, the primes $p_i$ and $q_j$ satisfy the
bounds~\eqref{eq:oo1} and~\eqref{eq:oo2}, and $m\in\cNnot$. The
set $\cA_{r,s}$ can be explicitly determined by a numerical
computation, and we obtain a finite list of ``admissible'' values
for the quantity $s(n)$.

To determine explicitly all of the numbers $n\in\cS\cap\cNnot$
with $r=\omega_\cP(n)$ and $s=\omega_\cQ(n)$, for every
$m\in\cA_{r,s}$ we need to find all such numbers for which
$s(n)=m$. To do this, factor $m$ as in~\eqref{eq:factorm}. For
fixed $i=1,\ldots,r$, let $\alpha_i$ be the largest integer such
that the number $mp_i^{\alpha_i-1}$ lies in $\cNnot$.  Similarly,
for fixed $j=1,\ldots,s$, let $\beta_j$ be the largest integer
such that the number $mq_j^{\beta_j-1}$ lies in $\cNnot$.  Put
$$
M=m\cdot p_1^{\alpha_1-1}\cdots p_r^{\alpha_r-1}
q_1^{\beta_1-1}\cdots q_s^{\beta_s-1}.
$$
Then, it is easy to see that $m\mid n$ and $n\mid M$ for any
$n\in\cS\cap\cNnot$ such that $s(n)=m$.  Hence, $n$ can take only
finitely many values which can be determined explicitly for each
$m\in\cA_{r,s}$.

For example, taking $r=s=2$ we find that
\begin{equation*}
\begin{split}
\{&4410, 8820, 10890, 13230, 17640, 21780, 22050, 26460, 30870,
35280,39690,\\
& 44100, 52920, 61740, 66150, 70560, 79380, 88200, 92610, 105840,
110250\}
\end{split}
\end{equation*}
is a complete list of the numbers $n\in\cS\setminus\cN$ with
$\omega_\cP(n)=\omega_\cQ(n)=2$.  Examining the lists generated as
$(r,s)$ varies over the pairs in $\cX$, we are lead to the
statement of Theorem~\ref{thm:main2}.

\section{Evaluation of $\limsup\limits_{n\in S}\frac{n}{\varphi(n)\log\log
n}$ and $\limsup\limits_{n\in S}\frac{\sigma(n)}{n\log\log n}$}

We conclude the paper by giving two propositions and two
corollaries that yield the analogue of the work of
Landau~\cite{Land} and Gronwall~\cite{Gron} for any set $\cS$ of
the form~\eqref{eq:S-defn} and for the set of natural numbers
equal to a sum of two squares. In fact,
Corollary~\ref{cor:LandGron} shows that Theorem~\ref{thm:main} is
nontrivial in the sense that $F(n)/\log\log n$ cannot be bounded
away from $e^\gamma$ by any positive constant for all large
$n\in\cS$. We will use the notation $f(n)=o(g(n))$ to mean that
$\lim\limits_{n\to\infty}f(n)/g(n)=0$.

\begin{proposition}
\label{prop:limsig} Let $\{a_n \}$ be an infinite sequence of
positive integers such that if we write $a_n=\prod_p p^{v(p,n)}$
we have:
\begin{itemize}
\item[$(i)$] $\kappa(a_n)=\prod_{p\le n}p$\quad $($i.e.,
$v(p,n)=0~\Longleftrightarrow~p>n)$;

\item[$(ii)$] $a_n=\exp(n^{1+o(1)})$;

\item[$(iii)$] $\lim\limits_{n\to\infty}v(p,n)=\infty$ for each
$p$.
\end{itemize}
Then,
$$
\lim_{n\rightarrow\infty}\frac{\sigma(a_n)}{a_n \log\log
a_n}=e^\gamma.
$$
\end{proposition}

\begin{proof}
For all $n\ge 1$, let
$$
b_n=\prod_{p\le n}p\mand c_n=\frac{\sigma(a_n)}{a_n}
\,\frac{\varphi(b_n)}{b_n},
$$
and observe that $(i)$ implies
$$
c_n=\biggl(\,\prod_{p\le
n}\frac{p^{v(p,n)+1}-1}{p^{v(p,n)}(p-1)}\biggl)\biggl(\,\prod_{p\le
n}\frac{p-1}{p}\biggl)=\prod_{p\le n}\(1-\frac{1}{p^{v(p,n)+1}}\).
$$
Since $v(p,n)+1\ge 2$ for every prime $p\le n$, we have for any
$m\le n$:
$$
1\ge c_n>\prod_{p\le
m}\(1-\frac{1}{p^{v(p,n)+1}}\)\prod_{p>m}\(1-\frac{1}{p^2}\).
$$
Using $(iii)$ we have for every fixed integer $m$:
$$
1\ge\limsup_{n\to\infty}c_n\ge\liminf_{n\to\infty}c_n\ge
\prod_{p>m}\(1-\frac{1}{p^2}\).
$$
The product on the right tends to one as $m\to\infty$, hence
$\lim_{n\to\infty}c_n=1$; therefore,
$$
\lim_{n\to\infty}\frac{\sigma(a_n)}{a_n\log n}=
\lim_{n\to\infty}\frac{b_n}{\varphi(b_n)\log n}\,.
$$
Our assumption $(ii)$ implies that $\log\log a_n=(1+o(1))\log n$,
and using Mertens' theorem (see, for example,~\cite{RosSch}) we
have
$$
\frac{\varphi(b_n)}{b_n}=\prod_{p\le
n}\(1-\frac{1}{p}\)=(1+o(1))\,\frac{e^{-\gamma}}{\log n},
$$
and the result follows.
\end{proof}

Using similar ideas (and an easier argument) one can obtain the
following analogue of Proposition~\ref{prop:limsig} for the Euler
totient function:

\begin{proposition}
\label{prop:limphi} Let $\{a_n \}$ be an infinite sequence of
positive integers such that:
\begin{itemize}
\item[$(i)$]  $\kappa(a_n)=\prod_{p\le n}p$;

\item[$(ii)$] $a_n=\exp(n^{1+o(1)})$.
\end{itemize}
Then,
$$
\lim_{n\rightarrow\infty}\frac{a_n}{\varphi(a_n)\log\log
a_n}=e^\gamma.
$$
\end{proposition}

\begin{corollary}
\label{cor:LandGron} For any set $\cS$ defined
by~\eqref{eq:S-defn}, we have
$$
\limsup_{n\in\cS}\frac{\sigma(n)}{n\log\log
n}=\limsup_{n\in\cS}\frac{n}{\varphi(n)\log\log n}=e^\gamma.
$$
\end{corollary}

\begin{proof}
Since
$$
\limsup_{n\to\infty}\frac{\sigma(n)}{n\log\log
n}=\limsup_{n\to\infty}\frac{n}{\varphi(n)\log\log n}=e^\gamma
$$
by~\cite{Gron} and~\cite{Land}, respectively, it suffices to show
that there is a sequence $\{a_n\}$ in $\cS$ such that
$$
\lim_{n\to\infty}\frac{\sigma(a_n)}{a_n\log\log
a_n}=\lim_{n\to\infty}\frac{a_n}{\varphi(a_n)\log\log
a_n}=e^\gamma.
$$
Let $a_1=1$, and for every integer $n\ge 2$, let
$$
b_n=\prod_{p\le n}p,\qquad d_n=\big\lfloor n^{(\log
n)^{-1/2}}\big\rfloor\mand a_n=b_n^{d_n}.
$$
It is easy to see that $d_n\ge 2$ for $n\ge 2$, $d_n=n^{o(1)}$,
and $d_n$ tends to infinity with $n$. Clearly, $a_n\in\cS$ for all
$n\ge 1$, and by the Prime Number Theorem in the form $\sum_{p\le
x}\log p=x(1+o(1))$ as $x\to\infty$ we see that
$$
\log a_n=d_n\log b_n=n^{o(1)}\sum_{p\le n}\log
p=n^{1+o(1)}\qquad(n\to\infty).
$$
The sequence $\{a_n\}$ therefore satisfies the hypotheses of
Propositions~\ref{prop:limsig} and~\ref{prop:limphi}, and the
result follows.
\end{proof}

\begin{corollary}
We have
$$\limsup_{n=a^2 +b^2}\frac{\sigma(n)}{n\log\log n}=\limsup_{n=a^2
+b^2}\frac{n}{\varphi(n)\log\log n}=e^\gamma.$$
\end{corollary}

\begin{proof}
Defining $a_n$ for all $n\ge 1$ as in the proof of
Corollary~\ref{cor:LandGron}, it is easy to see that the sequence
$\{a_n^2\}$ satisfies the hypotheses of
Propositions~\ref{prop:limsig} and~\ref{prop:limphi}; it follows
that
$$
\limsup_{n=a^2}\frac{\sigma(n)}{n\log\log
n}=\limsup_{n=a^2}\frac{n}{\varphi(n)\log\log n}=e^\gamma,
$$
and this implies the stated result.
\end{proof}

\end{document}